\documentclass{article}

\usepackage{amsmath}
\usepackage{amssymb}
\usepackage{graphicx}
\usepackage{bm}
\usepackage{mathdots}
\usepackage{booktabs}
\usepackage{rotating}
\usepackage{listings}
\usepackage[ruled,linesnumbered]{algorithm2e}
\usepackage[a4paper,left=2.8cm,right=2.8cm,top=2.5cm,bottom=2.5cm]{geometry}
\usepackage{fancyhdr}
\usepackage[framemethod=tikz]{mdframed}
\newtheorem{theorem}{Theorem}[section]
\newtheorem{corollary}{Corollary}[section]
\newtheorem{proposition}{Proposition}[section]

\newtheorem{lemma}{Lemma}[section]
\newtheorem{definition}{Definition}[section]

\makeatletter 
\@addtoreset{equation}{section}
\makeatother  

\newtheorem{remark}{Remark}[section]

\newenvironment{proof}{{\noindent\it Proof.}\quad}{\hfill $\square$\\}

\usepackage{url}
\usepackage{mathtools}
\usepackage{lipsum}
\usepackage{tcolorbox}

\newcommand{\eps}{\varepsilon}

\newcommand{\experimentfigure}[2]{%
\IfFileExists{#1}{%
\includegraphics[width=#2]{#1}%
}{%
\fbox{\parbox[c][0.24\textheight][c]{#2}{\centering
Figure file not found. Run \texttt{matlab/run\_all\_experiments.m}.}}%
}%
}

\begin{document}
\title{When do perturbed Chebyshev--Lobatto points remain Chebyshev?
}

\author{ Hao-Ning Wu\footnotemark[2]
       }

\renewcommand{\thefootnote}{\fnsymbol{footnote}}
\footnotetext[2]{Department of Mathematics, University of Georgia, Athens, GA 30602, USA (hnwu@uga.edu)}

\maketitle

\begin{abstract}
Chebyshev points are distinguished in polynomial interpolation by the
logarithmic growth of their Lebesgue constants. This paper asks a simple
question: how much can Chebyshev points be perturbed before they cease to
behave like Chebyshev points? We study perturbed Chebyshev--Lobatto nodes
\(x_j=\cos(j\pi/n+\eps_j)\), with angular perturbations
\(|\eps_j|\le \sigma_n\). The study is motivated by numerical experiments
showing a broad stable region when the mesh fraction \(n\sigma_n\) is small
and rapid amplification for larger perturbations; the observed transition
region is consistent with the curve \(n\sigma_n\asymp(\log n)^{-1}\). The main
result is a deterministic worst-case stability estimate: if
\(n\sigma_n(\log n+1)\) is bounded by a sufficiently small constant, then the
Lebesgue constant remains logarithmic. The proof uses the cosine
parametrization and Bernstein's inequality for trigonometric polynomials,
thereby exploiting the angular geometry of the Chebyshev--Lobatto grid rather
than a Markov inequality in the physical variable. We also give a worst-case
obstruction at the angular mesh scale, showing that perturbations of order
\(1/n\) cannot be allowed uniformly. Consequences are derived for analytic
interpolation in Bernstein ellipses, for the absence of Runge-type divergence
in the stable analytic regime, and for pseudospectral differentiation.
Numerical experiments illustrate the transition in the Lebesgue constants, the
shape of the associated Lebesgue functions, Runge-function interpolants, and
finite-precision differentiation errors.
\end{abstract}



\section{Introduction}

Polynomial interpolation at Chebyshev points is one of the most classical and successful tools in numerical approximation; see, e.g.,
\cite{Trefethen2013ATAP}. Let
\begin{equation}\label{equ:ChebyshevLobatto}
    x_j^{(0)} = \cos\frac{j\pi}{n},\qquad j=0,\ldots,n,
\end{equation}
be the Chebyshev--Lobatto points on \([-1,1]\). Their Lebesgue constant satisfies
\begin{equation}\label{equ:asymptotic}
\Lambda_n^{(0)}
    \leq
    \frac{2}{\pi}\log (n+1) + 1,\qquad 
    \Lambda_n^{(0)}
    =
    \frac{2}{\pi}\log n + O(1);
\end{equation} 
see \cite{MR194799,MR145246}. This logarithmic growth is essentially optimal among interpolation schemes using \(n+1\) points on \([-1,1]\). The question studied in this paper is deliberately elementary:
if the Chebyshev points are perturbed, when does the Lebesgue constant remain of order \(\log n\) in the sense of \eqref{equ:asymptotic}, and when does it begin to grow algebraically or worse? Since Chebyshev points cluster near \(\pm1\), a natural way to perturb Chebyshev points is not directly in the \(x\)-variable but the angular variable. We then consider the following perturbation model.
\begin{definition}[Perturbed Chebyshev--Lobatto points]
Let \(n\geq1\). For \(j=0,\ldots,n\), define
\[
    \theta_j^{(0)}=\frac{j\pi}{n}.
\]
A set of perturbed Chebyshev--Lobatto points is given by
\begin{equation}\label{equ:perturbedChebyshevLobatto}
    \theta_j=\theta_j^{(0)}+\eps_j,
    \qquad
    x_j=\cos\theta_j,
\end{equation}
where
\[
    \theta_0=0,\qquad \theta_n=\pi,
\]
and the interior perturbations satisfy
\[
    |\eps_j|\leq \sigma_n,
    \qquad j=1,\ldots,n-1.
\]
\end{definition}
The Chebyshev--Lobatto points are exactly equally spaced in \(\theta\). Thus the natural mesh width is
\[
    h_n=\frac{\pi}{n},
\]
which is of order \(n^{-1}\). The natural dimensionless perturbation parameter
is therefore the mesh fraction \(n\sigma_n\). Perturbations with
\(n\sigma_n=O(1)\) are of the same order as the angular mesh, while
\(n\sigma_n\ll1\) are smaller than the mesh. The analysis in this paper singles out
the more restrictive small-constant condition
\[
    n\sigma_n(\log n+1)\le \eta,
    \qquad \eta>0 \hbox{ sufficiently small},
\]
which is the regime in which we can prove deterministic logarithmic stability
of the Lebesgue constant. Throughout the paper, \(\log\) denotes the natural
logarithm.

\begin{figure}[htbp]
  \centering
  \includegraphics[width=0.72\textwidth]{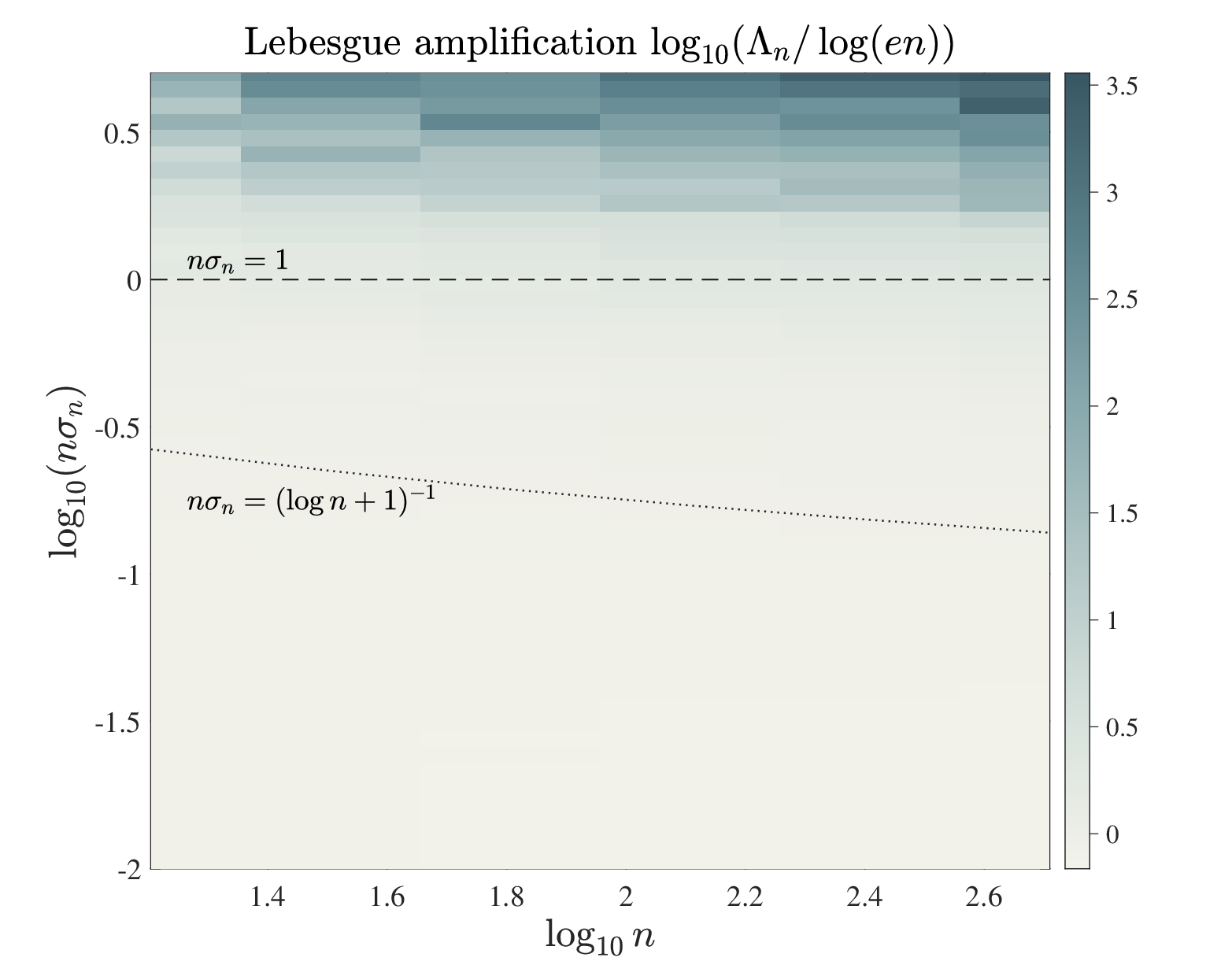}
  \caption{Numerical motivation. The colour indicates
  \(\log_{10}(\Lambda_n/\log(en))\) for random angular perturbations with
  \(|\varepsilon_j|\le \sigma_n\), plotted against the degree \(n\) and the
  mesh fraction \(n\sigma_n\). The dashed line marks the mesh scale
  \(n\sigma_n=1\), and the dotted line marks the deterministic stability scale
  \(n\sigma_n=(\log n+1)^{-1}\).}
  \label{fig:motivation}
\end{figure}

Figure~\ref{fig:motivation} shows the effect of varying the mesh fraction
\(n\sigma_n\). The plotted quantity is not \(\Lambda_n\) itself, but
\(\Lambda_n/\log(en)\), so that values of order one correspond to Chebyshev
scale growth. A stable region is visible for small \(n\sigma_n\), while
larger mesh fractions lead to much stronger amplification. The natural mesh
scale \(n\sigma_n=1\) is marked for orientation, but the numerics suggest a
more subtle boundary: the transition layer is roughly parallel to the curve
\[
    n\sigma_n=(\log n+1)^{-1}.
\]
This observation motivates the deterministic estimate proved in this paper.
We do not claim that this curve is the exact probabilistic threshold for random
perturbations. Rather, it identifies the small-constant scale at which a
worst-case argument guarantees that the Lebesgue constant remains logarithmic.

Two related lines of work should be mentioned. Austin and Trefethen
\cite{MR3693604} studied trigonometric interpolation and quadrature on
perturbed equispaced periodic grids, where fixed mesh-fraction perturbations
can be allowed. The present paper concerns algebraic interpolation at
Chebyshev--Lobatto points. The relevant geometry is no longer the periodic
grid itself, but the cosine map from a uniform angular grid to the interval.
The result is therefore local in the mesh fraction: we prove deterministic
logarithmic stability for angular perturbations of size
\(O((n\log n)^{-1})\), with a sufficiently small constant.

The work is also related to stability inequalities for norming sets, especially
those of Piazzon and Vianello~\cite{PiazzonVianello2018}. Their Markov-type
framework gives a general perturbation principle for Lebesgue constants. For
Chebyshev--Lobatto points, the angular parametrization improves the relevant
one-dimensional scale: the usual Markov factor \(n^2\) in the physical
variable is replaced by the trigonometric Bernstein factor \(n\) for
\(p(\cos\theta)\). This distinction, and the relation between the two
perturbation scales, is discussed after the main theorem.

The paper is organized as follows. Section~\ref{sec:interpolation} recalls the
interpolation notation and the operator interpretation of the Lebesgue
constant. Section~\ref{sec:main-theorem} proves the deterministic stability
estimate and gives a near-collision obstruction at the angular mesh scale.
Section~\ref{sec:runge-bernstein} translates the logarithmic stability bound
into analytic interpolation estimates and a Runge-type interpretation.
Section~\ref{sec:pseudospectral} records the corresponding consequence for
pseudospectral differentiation, and Section~\ref{sec:numerical} presents the
numerical experiments.

\section{Interpolation and Lebesgue constants}
\label{sec:interpolation}

Let \(\mathcal{X}_n = \{x_0,\dots,x_n\} \subset [-1,1]\) be a set of distinct nodes and let $\{f(x_j)\}_{j=0}^n$ be function values sampled on $\mathcal{X}_n$. The Lagrange interpolation is defined by
\[
I_n f(x) = \sum_{j=0}^n f(x_j)\, \ell_j(x),
\]
where the Lagrange basis polynomials are given by
\[
\ell_j(x) = \prod_{\substack{k=0,\, k\neq j}}^n \frac{x-x_k}{x_j-x_k}, \qquad j=0,\dots,n.
\]
A key quantity controlling the stability of this interpolation is the Lebesgue function,
\[
\lambda_n(x) = \sum_{j=0}^n |\ell_j(x)|,
\]
whose maximum over \([-1,1]\) is the Lebesgue constant
\begin{equation}\label{equ:Lebesgue}
\Lambda_n = \max_{x\in[-1,1]}\lambda_n(x).
\end{equation}
The Lebesgue constant $\Lambda_n$ measures how much errors in the data \(f(x_j)\) are amplified in the interpolant. Indeed, the interpolation error satisfies the classical inequality
\[
\|f-I_n f\|_\infty \le (1+\Lambda_n)\, E_n(f),
\]
where 
\[E_n(f) = \inf_{p\in \mathbb{P}_n} \|f-p\|_\infty\] 
denotes the best uniform approximation error; see, e.g., \cite[Chapter 15]{Trefethen2013ATAP}. Hence, the Lebesgue constant \(\Lambda_n\) quantifies the departure of the interpolation from the best uniform polynomial approximation.

\section{Lebesgue constants under angular perturbations}
\label{sec:main-theorem}

In this section we prove a deterministic stability theorem for the
Lebesgue constant under small angular perturbations of the
Chebyshev--Lobatto points. The argument is based on a simple comparison
between polynomial data sampled at the perturbed nodes and polynomial data
sampled at the unperturbed Chebyshev--Lobatto nodes.

Recall that 
\[
    x_j^{(0)}=\cos\theta_j^{(0)},
    \qquad
    \theta_j^{(0)}=\frac{j\pi}{n},
    \qquad
    j=0,\ldots,n,
\]
denote unperturbed Chebyshev--Lobatto nodes and 
\[
    x_j=\cos\theta_j,
    \qquad
    \theta_j=\theta_j^{(0)}+\varepsilon_j,
    \qquad
    j=0,\ldots,n,
\]
are perturbed ones. We assume throughout this section that the perturbed nodes \(x_0,\ldots,x_n\)
are distinct. Let \(\Lambda_n^{(0)}\) denote the Lebesgue constant of the
Chebyshev--Lobatto points \(x_j^{(0)}\), and let \(\Lambda_n\) denote the
Lebesgue constant of the perturbed points \(x_j\).

\subsection{A logarithmic stability bound}

We first consider an operator-norm characterization of the Lebesgue constant.

\begin{lemma}\label{lem:characterization}
For a
set of distinct interpolation nodes \(\mathcal{X}_n=\{x_0,\ldots,x_n\}\subset[-1,1]\),
let 
\[\Lambda(\mathcal{X}_n)=\max_{x\in[-1,1]}\sum_{j=0}^n |\ell_j(x)|\]
denote the corresponding Lebesgue constant. Then
\begin{equation}\label{equ:characterization}
    \Lambda(\mathcal{X}_n)
    =
    \sup_{p\in\mathbb{P}_n,\,
        \max_{0\le j\le n}|p(x_j)|\le 1}
        \|p\|_{L^\infty[-1,1]}.
\end{equation}
\end{lemma}
\begin{proof}
For  every
\(p\in\mathbb{P}_n\),
\[
    p(x)=\sum_{j=0}^n p(x_j)\ell_j(x),
\]
and therefore
\[
    |p(x)|
    \le
    \max_{0\le j\le n}|p(x_j)|
    \sum_{j=0}^n |\ell_j(x)|.
\]
Taking the maximum over \(x\in[-1,1]\) gives
\[
    \|p\|_\infty
    \le
    \Lambda(\mathcal{X}_n)\max_{0\leq j\leq n}|p(x_j)|.
\]
Conversely, let \(x^*\in[-1,1]\) be a point where the Lebesgue
function attains its maximum, and define
\[
    a_j=\operatorname{sgn}\ell_j(x^*),\qquad 0\le j\le n,
\]
with an arbitrary choice when \(\ell_j(x^*)=0\). Let
\[
    p^*(x)=\sum_{j=0}^n a_j\ell_j(x).
\]
Then \(p^*\in\mathbb P_n\) and
\[
    \max_{0\le j\le n}|p^*(x_j)|=\max_{0\le j\le n}|a_j|\le 1.
\]
Moreover,
\[
    |p^*(x^*)|
    =
    \left|\sum_{j=0}^n a_j\ell_j(x^*)\right|
    =
    \sum_{j=0}^n |\ell_j(x^*)|
    =
    \Lambda(\mathcal X_n).
\]
Hence
\[
    \sup_{p\in\mathbb{P}_n,\,
        \max_{0\le j\le n}|p(x_j)|\le 1}
        \|p\|_{L^\infty[-1,1]}
    \ge \Lambda(\mathcal X_n).
\]
Together with the previous inequality, the desired characterization \eqref{equ:characterization} follows.
\end{proof}

\begin{theorem}[Deterministic stability under angular perturbations]
\label{thm:main-angular-stability}
Assume that the perturbed nodes are distinct and that the angular perturbations
are uniformly bounded by
\[
    \max_{0\le j\le n}|\varepsilon_j|\le \sigma_n,
\]
where \(\sigma_n>0\) denotes the maximal magnitude of the perturbation in the
angular variable. If the magnitude $\sigma_n$ satisfies
\[
    n\sigma_n\Lambda_n^{(0)}<1,
\]
then the Lebesgue constant of the perturbed nodes $\{x_j\}_{j=0}^n$ satisfies
\begin{equation}\label{equ:mainestimate}
    \Lambda_n
    \le
    \frac{\Lambda_n^{(0)}}{1-n\sigma_n\Lambda_n^{(0)}}.
\end{equation}

\end{theorem}

\begin{proof}
Let \(p\in\mathbb{P}_n\) satisfy
\[
    \max_{0\le j\le n}|p(x_j)|\le 1.
\]
We shall bound \(\|p\|_\infty\) in terms of the Lebesgue constant \(\Lambda_n^{(0)}\) of unperturbed nodes. Define
\[
    q(\theta):=p(\cos\theta),
    \qquad \theta\in\mathbb R.
\]
Since \(p\) has degree at most \(n\), \(q\) is a trigonometric polynomial of
degree at most \(n\). Hence Bernstein's inequality for trigonometric
polynomials gives
\[
    \|q'\|_{L^\infty(\mathbb R)}
    \le
    n\|q\|_{L^\infty(\mathbb R)}.
\]
Since $\|q\|_{L^\infty(\mathbb R)}
    =
    \|p\|_{L^\infty[-1,1]}$,
we obtain
\[
    \|q'\|_{L^\infty(\mathbb R)}
    \le
    n\|p\|_{L^\infty[-1,1]}.
\]
For each \(j=0,\ldots,n\), the mean value theorem gives
\[
\begin{aligned}
    |p(x_j^{(0)})|
    &=
    |q(\theta_j^{(0)})| \le
    |q(\theta_j)|
    +
    |q(\theta_j^{(0)})-q(\theta_j)|  \\
    &\le
    |p(x_j)|
    +
    \|q'\|_{L^\infty[0,\pi]}
    |\theta_j^{(0)}-\theta_j|  \\
    &\le
    1
    +
    n\sigma_n
    \|p\|_{L^\infty[-1,1]}.
\end{aligned}
\]
Therefore
\begin{equation}\label{equ:previousestimate}
    \max_{0\le j\le n}|p(x_j^{(0)})|
    \le
    1+n\sigma_n\|p\|_{L^\infty[-1,1]}.
\end{equation}
Since \(p\in\mathbb{P}_n\), it is exactly the polynomial interpolant of its values
at the Chebyshev--Lobatto points. By the definition of \(\Lambda_n^{(0)}\) and Lemma \ref{lem:characterization},
\[
    \|p\|_{L^\infty[-1,1]}
    \le
    \Lambda_n^{(0)}
    \max_{0\le j\le n}|p(x_j^{(0)})|.
\]
Combining this with the previous estimate \eqref{equ:previousestimate} gives
\[
    \|p\|_{L^\infty[-1,1]}
    \le
    \Lambda_n^{(0)}
    \left(
        1+n\sigma_n\|p\|_{L^\infty[-1,1]}
    \right).
\]
Equivalently,
\[
    \left(1-n\sigma_n\Lambda_n^{(0)}\right)
    \|p\|_{L^\infty[-1,1]}
    \le
    \Lambda_n^{(0)}.
\]
Since \(n\sigma_n\Lambda_n^{(0)}<1\), we  obtain
\[
    \|p\|_{L^\infty[-1,1]}
    \le
    \frac{\Lambda_n^{(0)}}{1-n\sigma_n\Lambda_n^{(0)}}.
\]
This bound holds for every \(p\in\mathbb{P}_n\) satisfying
\(\max_j|p(x_j)|\le1\). By the operator-norm characterization of the
Lebesgue constant in Lemma \ref{lem:characterization}, the desired estimate \eqref{equ:mainestimate} follows.
\end{proof}

\begin{theorem}[Explicit logarithmic stability regime]
\label{thm:explicit-log-stability}
Let
\[
    x_j=\cos\left(\frac{j\pi}{n}+\varepsilon_j\right),
    \qquad j=0,\ldots,n,
\]
be distinct perturbed Chebyshev--Lobatto nodes, with
\(\varepsilon_0=\varepsilon_n=0\) and
\(\max_j|\varepsilon_j|\le\sigma_n\).
Let $\eta$ be a constant satisfying
\[0<\eta< \left(1+\frac{2}{\pi}\right)^{-1}\approx 0.611015,\]
and assume that
 \begin{equation}\label{equ:explicit-stability-assumption}
    n\sigma_n\left(\log n+1\right)\le \eta.
\end{equation}
Then the Lebesgue constant $\Lambda_n$ of the perturbed nodes satisfies
\[
    \Lambda_n
    \le
    C_\eta\left(\log{n} +1\right),
\]
where
\[
    C_\eta
    =
    \dfrac{1+\dfrac{2}{\pi}}
         {1-\left(1+\dfrac{2}{\pi}\right)\eta}.
\]
In particular, the perturbed points have logarithmically growing Lebesgue
constants whenever
\[
    n\sigma_n\log n\to0.
\]
\end{theorem}

\begin{proof}
We use the explicit estimate
\[
    \Lambda_n^{(0)}
    \le
    \frac{2}{\pi}\log(n+1)+1
\]
 for the Chebyshev--Lobatto Lebesgue constant.
Since \(n+1\le en\) and \(1\le \log(en)\) for \(n\ge1\), we have $\log(n+1)\le \log(en)$ and
\[
    \Lambda_n^{(0)}
    \le
    \frac{2}{\pi}\log(en)+\log(en)
    =
    \left(1+\frac{2}{\pi}\right)\left(\log{n} +1\right).
\]
It follows from the assumption \eqref{equ:explicit-stability-assumption} that
\[
    n\sigma_n\Lambda_n^{(0)}
    \le
    \left(1+\frac{2}{\pi}\right)n\sigma_n(\log n +1)
    \le
    \left(1+\frac{2}{\pi}\right)\eta
    <1.
\]
Hence the deterministic stability estimate \eqref{equ:mainestimate} of
Theorem~\ref{thm:main-angular-stability} applies and gives
\[
    \Lambda_n
    \le
    \frac{\Lambda_n^{(0)}}
         {1-n\sigma_n\Lambda_n^{(0)}}.
\]
Using the two bounds above, we obtain
\[
    \Lambda_n
    \le
    \frac{
    \left(1+\frac{2}{\pi}\right)\log(en)
    }{
    1-\left(1+\frac{2}{\pi}\right)\eta
    }.
\]
Thus
\[
    \Lambda_n
    \le
    C_\eta\log(en),
    \qquad
    C_\eta
    =
    \frac{1+\frac{2}{\pi}}
         {1-\left(1+\frac{2}{\pi}\right)\eta}.
\]
Finally, if \(n\sigma_n\log n\to0\), then for every fixed
\[
    0<\eta<\left(1+\frac{2}{\pi}\right)^{-1},
\]
the condition \(n\sigma_n\log(en)\le\eta\) holds for all sufficiently large
\(n\). Hence \(\Lambda_n=O(\log n)\).
\end{proof}

\begin{remark}
Theorem \ref{thm:explicit-log-stability} gives a deterministic stability regime. It does not claim to
identify a sharp threshold for random perturbations.
Rather, it shows that angular perturbations of size
\[
    \sigma_n=o\!\left(\frac{1}{n\log n}\right)
\]
cannot destroy the logarithmic growth of the Chebyshev--Lobatto Lebesgue
constant. The numerical experiments suggest that typical random perturbations
may remain stable in a larger regime, but this sharper probabilistic question
is separate from the deterministic estimate proved here.
\end{remark}

\begin{remark}[Angular versus \(x\)-variable perturbations]
The scale in Theorem~\ref{thm:explicit-log-stability} should be interpreted
in the angular variable. It is not the same as a uniform perturbation scale in
the physical variable \(x\). This is the point at which the present result
differs from the general norming-set perturbation principle of
Piazzon and Vianello~\cite{PiazzonVianello2018}. In that framework, stability
is obtained from a Lebesgue constant and a Markov-type constant for the
underlying polynomial space. On the interval, this uses the classical Markov
inequality
\[
    \|p'\|_\infty\le n^2\|p\|_\infty,
\]
and therefore gives, for the Chebyshev--Lobatto grid, perturbations satisfying
\[
    \max_j |x_j-x_j^{(0)}|
    =
    O\!\left(\frac{1}{n^2\log n}\right)
\]
with a sufficiently small constant, because
\(\Lambda_n^{(0)}=O(\log n)\). This is the usual interval scale obtained from
a Markov-type perturbation argument.

By contrast, the angular parametrization gives
\[
    x=\cos\theta,
    \qquad
    q(\theta)=p(\cos\theta),
\]
and therefore the trigonometric Bernstein inequality gives only a factor
\(n\). Consequently, angular perturbations of size
\[
    \sigma_n=O\!\left(\frac{1}{n\log n}\right)
\]
are covered when the implicit constant is sufficiently small. Thus the
improvement is not a contradiction of the Markov--norming-set bound; it comes
from measuring perturbations in the Chebyshev angular variable. In the
\(x\)-variable such perturbations are nonuniform: they are of order
\(O((n^2\log n)^{-1})\) near the endpoints and of order
\(O((n\log n)^{-1})\) in the interior. This is the geometry one expects for
perturbations that preserve the Chebyshev clustering.

\end{remark}

\subsection{An obstruction at the mesh scale}

The preceding estimate gives a deterministic stability guarantee at the
small-constant scale
\[
    \sigma_n = O\!\left(\frac{1}{n\log n}\right).
\]
It is natural to ask whether this scale can be enlarged all the way to the
angular mesh width \(h_n=\pi/n\).  This is the first larger scale one would
try, since \(h_n\) is the spacing of the unperturbed nodes in the angular
variable.  The following elementary construction shows that arbitrary
perturbations at the mesh scale cannot preserve Chebyshev-type stability.  The
obstruction is a near collision of two neighbouring angular nodes.

\begin{proposition}[A near-collision obstruction]
Let \(h_n=\pi/n\), and let \(\tau_n\in(0,1/4)\).  For every \(n\ge 4\), there
exists a perturbed Chebyshev--Lobatto set
\[
    x_j=\cos\theta_j,\qquad j=0,\ldots,n,
\]
with
\[
    \max_{0\le j\le n}|\theta_j-jh_n|\le \frac{h_n}{2},
\]
such that its Lebesgue constant satisfies
\[
    \Lambda_n \ge \frac{1}{\pi n\tau_n}.
\]
Consequently, by taking for example
\[
    \tau_n=\frac{1}{n(\log n)^2},
\]
one obtains a sequence of angular perturbations of size \(O(1/n)\) for which
\[
    \Lambda_n \ge \frac{1}{\pi}(\log n)^2.
\]
In particular, arbitrary mesh-scale angular perturbations are not covered by a
logarithmic stability principle.  They can lead to Lebesgue constants growing
faster than \(\log n\).
\end{proposition}

\begin{proof}
Set \(h_n=\pi/n\), and choose \(m=\lfloor n/2\rfloor\).  We perturb only the two
neighbouring angular nodes \(m h_n\) and \((m+1)h_n\).  Define
\[
    \theta_m=\left(m+\frac12\right)h_n-\tau_n h_n,
    \qquad
    \theta_{m+1}=\left(m+\frac12\right)h_n+\tau_n h_n,
\]
and set \(\theta_j=jh_n\) for all \(j\ne m,m+1\).  Then the angular perturbations
satisfy
\[
    |\theta_m-mh_n|=|\theta_{m+1}-(m+1)h_n|
    =\left(\frac12-\tau_n\right)h_n
    \le \frac{h_n}{2},
\]
and all nodes are distinct.

Let \(x_j=\cos\theta_j\), and let \(p_n\in\mathbb P_n\) be the interpolation
polynomial defined by the nodal values
\[
    p_n(x_m)=1,\qquad p_n(x_{m+1})=-1,\qquad
    p_n(x_j)=0\quad (j\ne m,m+1).
\]
Thus
\[
    \max_{0\le j\le n}|p_n(x_j)|\le 1.
\]
By the operator-norm characterization of the Lebesgue constant in Lemma \ref{lem:characterization},
\[
    \Lambda_n\ge \|p_n\|_{L^\infty[-1,1]}.
\]

On the other hand, by the mean value theorem,
\[
    \|p_n'\|_{L^\infty[-1,1]}
    \ge
    \frac{|p_n(x_m)-p_n(x_{m+1})|}{|x_m-x_{m+1}|}
    =
    \frac{2}{|x_m-x_{m+1}|}.
\]
Markov's inequality on \([-1,1]\) gives
\[
    \|p_n'\|_{L^\infty[-1,1]}
    \le
    n^2 \|p_n\|_{L^\infty[-1,1]}.
\]
Therefore
\[
    \|p_n\|_{L^\infty[-1,1]}
    \ge
    \frac{2}{n^2 |x_m-x_{m+1}|}.
\]
It remains to estimate the distance between the two nearly colliding nodes.
Since the cosine function is Lipschitz with constant one,
\[
    |x_m-x_{m+1}|
    =
    |\cos\theta_m-\cos\theta_{m+1}|
    \le
    |\theta_m-\theta_{m+1}|
    =
    2\tau_n h_n
    =
    \frac{2\pi\tau_n}{n}.
\]
Hence
\[
    \|p_n\|_{L^\infty[-1,1]}
    \ge
    \frac{2}{n^2(2\pi\tau_n/n)}
    =
    \frac{1}{\pi n\tau_n}.
\]
Since \(\Lambda_n\ge \|p_n\|_{L^\infty[-1,1]}\), the desired lower bound follows.
The choice \(\tau_n=1/(n(\log n)^2)\) gives
\[
    \Lambda_n\ge \frac{1}{\pi}(\log n)^2,
\]
which is superlogarithmic.  This completes the proof.
\end{proof}

\begin{remark}
The construction is deliberately adversarial: it uses an almost collision of two
neighbouring nodes.  It therefore does not show that the deterministic scale
\((n\log n)^{-1}\) is sharp.  Rather, it shows that the stability result cannot
be extended to arbitrary perturbations of order \(1/n\).  The intermediate
regime
\[
    \frac{1}{n\log n}\ll \sigma_n \ll \frac1n
\]
remains a separate question, and the numerical phase diagrams below suggest
that typical random perturbations may behave better than this worst-case
near-collision example.
\end{remark}

\section{Analytic interpolation and Runge-type behaviour}
\label{sec:runge-bernstein}

For \(\rho>1\), let \(E_\rho\) denote the open Bernstein ellipse with foci
\(\pm1\), whose boundary is the image of \(|w|=\rho\) under
\(z=(w+w^{-1})/2\). If \(f\) is analytic in \(E_\rho\) and continuous on
\(\overline E_\rho\), set
\[
    M_\rho=\max_{z\in \overline E_\rho}|f(z)|.
\]
We shall use the standard Bernstein estimate
\begin{equation}\label{equ:bernstein-best-approximation}
    E_n(f)
    :=
    \inf_{p\in\mathbb P_n}\|f-p\|_{L^\infty[-1,1]}
    \le
    \frac{2M_\rho}{\rho-1}\rho^{-n},
\end{equation}
see, for example, \cite{Trefethen2013ATAP,XiangChenWang2010}.

\begin{corollary}[Analytic interpolation at perturbed Chebyshev--Lobatto points]
\label{cor:no-runge-analytic}
Let \(f\) be analytic in \(E_\rho\), \(\rho>1\), and continuous on
\(\overline E_\rho\). For each \(n\), let \(I_n^\varepsilon f\) be the
interpolant at distinct perturbed Chebyshev--Lobatto nodes satisfying
\[
    x_j=\cos\left(\frac{j\pi}{n}+\varepsilon_j\right),
    \qquad
    \max_{0\le j\le n}|\varepsilon_j|\le \sigma_n .
\]
Assume that, for some fixed \(\eta\),
\[
    n\sigma_n(\log n+1)\le \eta,
    \qquad
    0<\eta<\left(1+\frac{2}{\pi}\right)^{-1}.
\]
Then
\begin{equation}\label{equ:analytic-interpolation-error}
    \|f-I_n^\varepsilon f\|_{L^\infty[-1,1]}
    \le
    \left(1+C_\eta\log(en)\right)
    \frac{2M_\rho}{\rho-1}\rho^{-n}.
\end{equation}
In particular, \(I_n^\varepsilon f\to f\) uniformly on \([-1,1]\).
\end{corollary}

\begin{proof}
The Lebesgue inequality gives
\[
    \|f-I_n^\varepsilon f\|_\infty
    \le
    (1+\Lambda_n)E_n(f).
\]
By Theorem~\ref{thm:explicit-log-stability},
\(\Lambda_n\le C_\eta\log(en)\). Together with
\eqref{equ:bernstein-best-approximation}, this proves
\eqref{equ:analytic-interpolation-error}. Since
\(\log(en)\rho^{-n}\to0\), the asserted convergence follows.
\end{proof}

\begin{corollary}[The Runge function under stable perturbations]
\label{cor:runge-free}
Let
\[
    f_R(x)=\frac{1}{1+25x^2}.
\]
Let \(I_n^\varepsilon f_R\) be defined as in
Corollary~\ref{cor:no-runge-analytic}. If
\[
    n\sigma_n(\log n+1)\le \eta,
    \qquad
    0<\eta<\left(1+\frac{2}{\pi}\right)^{-1},
\]
then, for every
\[
    1<\rho<\rho_R,
    \qquad
    \rho_R=\frac15+\sqrt{1+\frac1{25}},
\]
there is a constant \(M_\rho\) such that
\[
    \|f_R-I_n^\varepsilon f_R\|_{L^\infty[-1,1]}
    \le
    \left(1+C_\eta\log(en)\right)
    \frac{2M_\rho}{\rho-1}\rho^{-n}.
\]
Consequently \(I_n^\varepsilon f_R\to f_R\) uniformly on \([-1,1]\). In this
analytic setting, stable perturbed Chebyshev--Lobatto interpolation does not
exhibit Runge-type divergence.
\end{corollary}

\begin{proof}
The singularities of \(f_R\) are at \(\pm i/5\). Hence \(f_R\) is analytic in
\(E_\rho\) and continuous on \(\overline E_\rho\) for every
\(1<\rho<\rho_R\). The result is Corollary~\ref{cor:no-runge-analytic}.
\end{proof}

For equally spaced interpolation, the corresponding Lebesgue constants grow
exponentially; see \cite{Runge1901,PlatteTrefethenKuijlaars2011}. The preceding
corollaries isolate the different mechanism here: the factor \(1+\Lambda_n\)
is only logarithmic and cannot dominate the Bernstein decay \(\rho^{-n}\).

\section{Pseudospectral differentiation}
\label{sec:pseudospectral}

The preceding section shows that interpolation at the perturbed
Chebyshev--Lobatto points remains spectrally accurate for analytic functions
in the deterministic stability regime. We also record the corresponding
consistency consequence for pseudospectral differentiation. The underlying
mechanism is the same stability estimate in a differentiated form: a
pseudospectral derivative is obtained by differentiating the interpolation
polynomial, see, e.g., \cite{MR1776072}, and
Markov's inequality converts uniform stability of interpolation into stability
of any fixed number of derivatives, at the cost of a polynomial factor in
\(n\). This factor does not destroy the geometric decay supplied by analyticity
in a Bernstein ellipse. Error behaviour for differentiated Chebyshev spectral approximations is also
closely related to localization phenomena; see, e.g., \cite{Wang2023ErrorLocalization}.

Let
\[
    x_j=\cos\left(\frac{j\pi}{n}+\varepsilon_j\right),
    \qquad j=0,\ldots,n,
\]
and let \(I_n^\varepsilon\) denote the interpolation operator at these
perturbed nodes:
\[
    I_n^\varepsilon f\in \mathbb{P}_n,
    \qquad
    I_n^\varepsilon f(x_j)=f(x_j),
    \qquad j=0,\ldots,n.
\]
If \(\ell_j^\varepsilon\) are the corresponding Lagrange basis polynomials, then
\[
    I_n^\varepsilon f(x)
    =
    \sum_{j=0}^n f(x_j)\ell_j^\varepsilon(x),
\]
and the pseudospectral approximation of the \(m\)-th derivative is
\[
    (I_n^\varepsilon f)^{(m)}(x)
    =
    \sum_{j=0}^n f(x_j)(\ell_j^\varepsilon)^{(m)}(x).
\]
In particular, evaluating this expression at the nodes gives the usual
pseudospectral differentiation matrix
\[
    D_{ij}^{(m),\varepsilon}
    =
    (\ell_j^\varepsilon)^{(m)}(x_i),
    \qquad 0\le i,j\le n.
\]

\begin{theorem}[Perturbed pseudospectral differentiation]
\label{thm:pseudospectral-differentiation}
Assume that the nodes are distinct and that
\[
    \max_{0\le j\le n}|\varepsilon_j|\le \sigma_n .
\]
Suppose moreover that, for some fixed \(\eta\),
\[
    n\sigma_n(\log n+1)\le \eta,
    \qquad
    0<\eta<\left(1+\frac{2}{\pi}\right)^{-1}.
\]
Let \(f\) be analytic in the Bernstein ellipse \(E_\rho\), \(\rho>1\), and
continuous on \(\overline E_\rho\). Set
\[
    M_\rho=\max_{z\in\overline E_\rho}|f(z)|.
\]
Then, for every fixed integer \(m\ge1\), there exists a constant
\(C_{\rho,m,\eta}>0\), independent of \(n\), such that
\[
    \left\|
    f^{(m)}-(I_n^\varepsilon f)^{(m)}
    \right\|_\infty
    \le
    C_{\rho,m,\eta}
    M_\rho\, n^{2m}\log(en)\rho^{-n}.
\]
Consequently,
\[
    \max_{0\le i\le n}
    \left|
    f^{(m)}(x_i)
    -
    \sum_{j=0}^n f(x_j)(\ell_j^\varepsilon)^{(m)}(x_i)
    \right|
    \le
    C_{\rho,m,\eta}
    M_\rho\, n^{2m}\log(en)\rho^{-n}.
\]
In particular, for each fixed derivative order \(m\), pseudospectral
differentiation at the perturbed nodes is spectrally accurate.
\end{theorem}

\begin{proof}
The case \(m=0\) is exactly the interpolation estimate proved in
Corollary~\ref{cor:no-runge-analytic}. We therefore assume \(m\ge1\).

Let \(p_n\in \mathbb{P}_n\) be a polynomial approximant to \(f\) satisfying the
standard Bernstein estimates
\[
    \|f-p_n\|_\infty
    \le
    C_\rho M_\rho \rho^{-n},
\]
and, for the fixed integer \(m\),
\[
    \|f^{(m)}-p_n^{(m)}\|_\infty
    \le
    C_{\rho,m}M_\rho n^{2m}\rho^{-n}.
\]
For example, \(p_n\) may be taken to be the truncated Chebyshev expansion of
\(f\).

Since \(p_n\in \mathbb{P}_n\), interpolation is exact on \(p_n\):
\[
    I_n^\varepsilon p_n=p_n.
\]
Therefore
\[
    f-I_n^\varepsilon f
    =
    (f-p_n)-I_n^\varepsilon(f-p_n).
\]
Differentiating \(m\) times gives
\[
    f^{(m)}-(I_n^\varepsilon f)^{(m)}
    =
    \bigl(f^{(m)}-p_n^{(m)}\bigr)
    -
    \bigl(I_n^\varepsilon(f-p_n)\bigr)^{(m)}.
\]
Hence
\[
    \left\|
    f^{(m)}-(I_n^\varepsilon f)^{(m)}
    \right\|_\infty
    \le
    \|f^{(m)}-p_n^{(m)}\|_\infty
    +
    \left\|
    \bigl(I_n^\varepsilon(f-p_n)\bigr)^{(m)}
    \right\|_\infty .
\]

The first term is already bounded by
\[
    C_{\rho,m}M_\rho n^{2m}\rho^{-n}.
\]
For the second term, \(I_n^\varepsilon(f-p_n)\) is a polynomial of degree at
most \(n\). By Markov's inequality, for every \(r\in \mathbb{P}_n\),
\[
    \|r^{(m)}\|_\infty
    \le
    C_m n^{2m}\|r\|_\infty .
\]
Applying this to $I_n^\varepsilon(f-p_n)$,
we obtain
\[
    \left\|
    \bigl(I_n^\varepsilon(f-p_n)\bigr)^{(m)}
    \right\|_\infty
    \le
    C_m n^{2m}
    \|I_n^\varepsilon(f-p_n)\|_\infty.
\]

The operator norm of the interpolation operator is the Lebesgue constant.
Thus
\[
    \|I_n^\varepsilon g\|_\infty
    \le
    \Lambda_n\|g\|_\infty .
\]
By Theorem~\ref{thm:explicit-log-stability},
$\Lambda_n\le C_\eta\log(en)$.
Therefore
\[
\begin{aligned}
    \|I_n^\varepsilon(f-p_n)\|_\infty
    &\le
    \Lambda_n\|f-p_n\|_\infty  \le
    C_\eta\log(en)\,
    C_\rho M_\rho\rho^{-n}.
\end{aligned}
\]
Consequently,
\[
    \left\|
    \bigl(I_n^\varepsilon(f-p_n)\bigr)^{(m)}
    \right\|_\infty
    \le
    C_{\rho,m,\eta}
    M_\rho n^{2m}\log(en)\rho^{-n}.
\]
Combining this estimate with the bound for
\(\|f^{(m)}-p_n^{(m)}\|_\infty\) gives
\[
    \left\|
    f^{(m)}-(I_n^\varepsilon f)^{(m)}
    \right\|_\infty
    \le
    C_{\rho,m,\eta}
    M_\rho n^{2m}\log(en)\rho^{-n}.
\]
This proves the continuous pseudospectral differentiation estimate.

Finally, evaluating the same estimate at the nodes \(x_i\) gives
\[
    \max_{0\le i\le n}
    \left|
    f^{(m)}(x_i)
    -
    (I_n^\varepsilon f)^{(m)}(x_i)
    \right|
    \le
    C_{\rho,m,\eta}
    M_\rho n^{2m}\log(en)\rho^{-n}.
\]
Since
\[
    (I_n^\varepsilon f)^{(m)}(x_i)
    =
    \sum_{j=0}^n f(x_j)(\ell_j^\varepsilon)^{(m)}(x_i),
\]
this is exactly the differentiation-matrix estimate.
\end{proof}

\begin{remark}
Theorem~\ref{thm:pseudospectral-differentiation} concerns pseudospectral
differentiation of a given analytic function. It should not be read as a
complete convergence theorem for a pseudospectral collocation method applied
to a differential equation. For a full collocation scheme, one also has to
control the stability or invertibility of the resulting discrete differential
operator. The result here is a consistency statement: in the deterministic
stability regime, the perturbed interpolation grid itself does not destroy
spectral accuracy of the differentiated interpolant.
\end{remark}

\section{Numerical experiments}
\label{sec:numerical}

We close with numerical experiments illustrating the preceding estimates.

\subsection{Lebesgue constants and Lebesgue functions}

We first examine the Lebesgue constants and Lebesgue functions of perturbed
Chebyshev--Lobatto points. The random perturbations are generated in the
angular variable, as in \eqref{equ:perturbedChebyshevLobatto}, with
\[
    \varepsilon_0=\varepsilon_n=0,
    \qquad
    \varepsilon_j\sim {\rm Unif}[-\sigma_n,\sigma_n],
    \qquad j=1,\ldots,n-1,
\]
for some scale $\sigma_n$ to be specified in the following text.
For each randomly generated node set, the nodes are sorted increasingly before
the Lebesgue function is evaluated.

For the numerical evaluation we use the barycentric formula \cite{BerrutTrefethen2004}, not as an
additional assumption but as a stable way of evaluating the same Lagrange
interpolants and Lebesgue functions defined in Section~\ref{sec:interpolation}. Thus the Lagrange basis polynomials need not be formed
explicitly. If 
\[w_j=\prod_{k\ne j}(x_j-x_k)^{-1},\] 
then away from the nodes
the Lebesgue function is computed as
\[
    \lambda_n(x)
    =\sum_{j=0}^n|\ell_j(x)|
    =
    \frac{\displaystyle \sum_{j=0}^n
    \left|w_j/(x-x_j)\right|}
    {\displaystyle \left|\sum_{j=0}^n w_j/(x-x_j)\right|}.
\]
At the interpolation nodes we set \(\lambda_n(x_j)=1\). To avoid overflow in
the weights, the magnitudes are formed logarithmically,
\[
    \log |w_j|
    =
    -\sum_{k\ne j}\log |x_j-x_k|,
\]
and then shifted by their maximum before exponentiation. The signs alternate
after sorting the nodes. The Lebesgue constant is then approximated by
maximizing \(\lambda_n\) on a dense mixed grid consisting of a uniform grid in
\(x\) and a uniform grid in the angular variable \(x=\cos\theta\). For the
Lebesgue-function plots below, additional endpoint-refined grid points are
included in order to resolve the large peaks of the equally spaced nodes near
\(\pm1\).

\begin{figure}[htbp]
  \centering
  \includegraphics[width=0.72\textwidth]{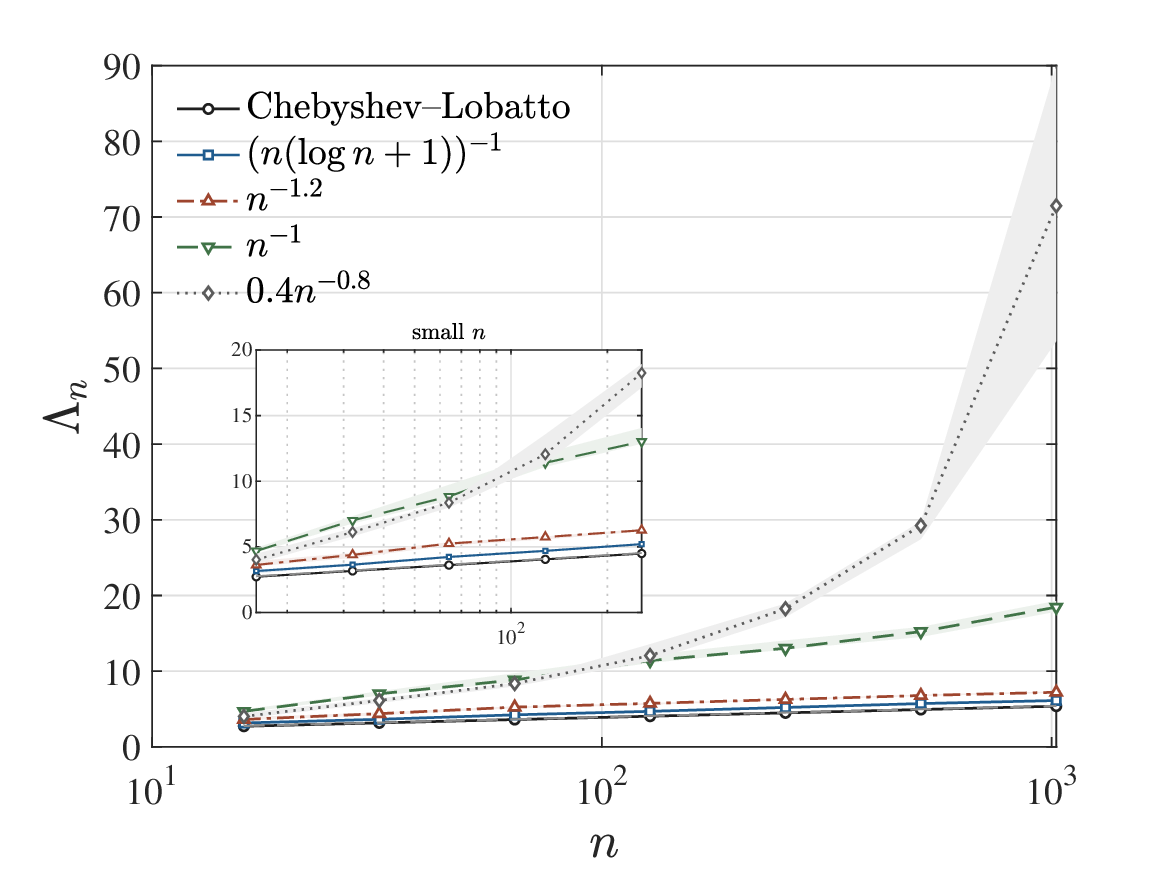}
  \caption{Lebesgue constants for perturbed Chebyshev--Lobatto points as
  functions of \(n\). The perturbation scales are
  \(\sigma_n=(n(\log n+1))^{-1}\), \(n^{-1.2}\), \(n^{-1}\), and
  \(0.4n^{-0.8}\). Random curves show medians over eight independent
  perturbations; shaded bands indicate the middle half of the samples.}
  \label{fig:lebesgue-vs-n}
\end{figure}

Figure~\ref{fig:lebesgue-vs-n} compares perturbations at the scale
\((n(\log n+1))^{-1}\), below it, at the angular mesh scale \(n^{-1}\), and
above the mesh scale. The theorem gives logarithmic stability when
\(n\sigma_n(\log n+1)\) is bounded by a sufficiently small constant. The plotted
curve with \(\sigma_n=(n(\log n+1))^{-1}\) has constant one in this scale, so it
is not covered by the explicit small-constant theorem, but it follows the
unperturbed Chebyshev--Lobatto curve closely in this experiment. The slightly
smaller scale \(n^{-1.2}\) also behaves in the same way. At the mesh scale
\(n^{-1}\), the growth is still moderate but visibly larger. By contrast, the larger perturbation
\(0.4n^{-0.8}\) produces much larger Lebesgue constants over the computed range.
This is the numerical counterpart of the analytical message: sufficiently small
angular perturbations preserve the Chebyshev mechanism, while larger
perturbations can create substantial amplification.

\begin{figure}[htbp]
  \centering
  \includegraphics[width=\textwidth]{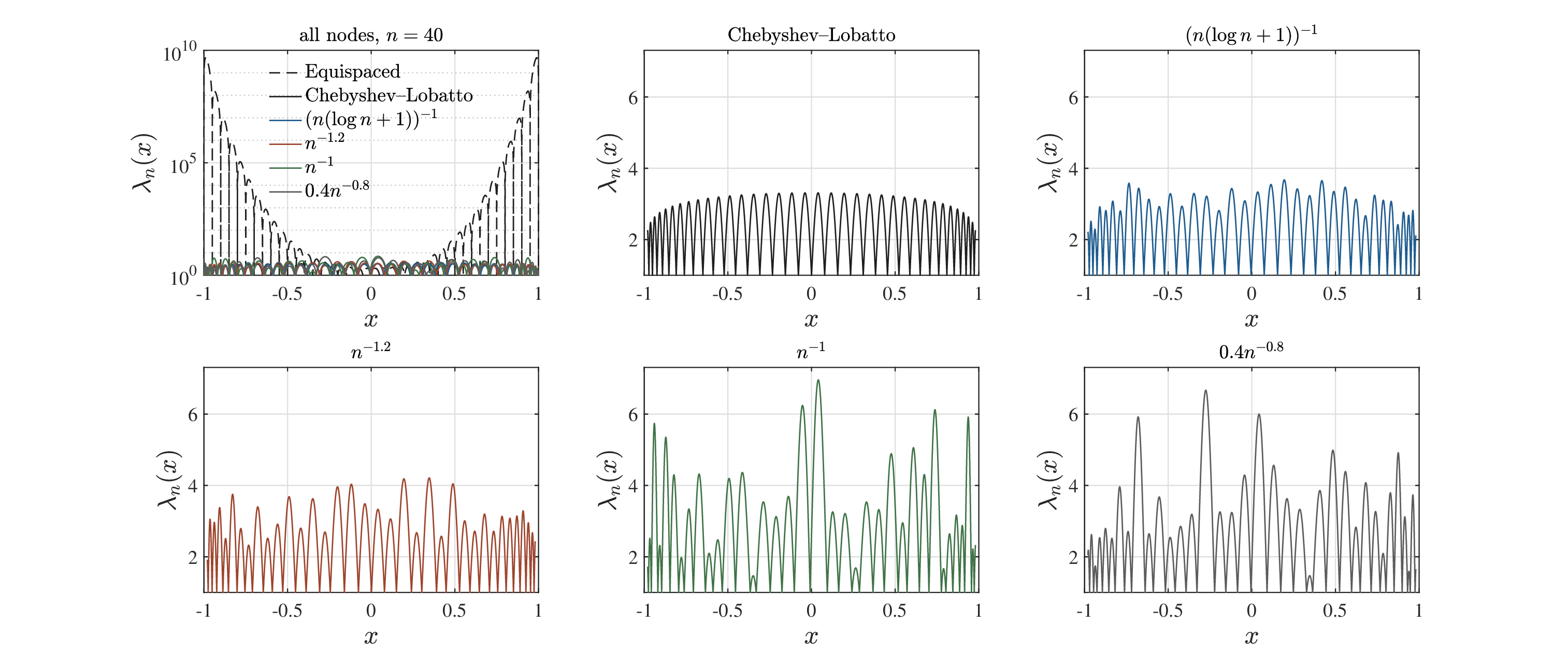}
  \caption{Lebesgue functions for selected node sets at fixed \(n\). The first
  panel compares equally spaced nodes, Chebyshev--Lobatto nodes, and four
  perturbed Chebyshev--Lobatto node sets. The equally spaced curve is shown
  with a dashed line. The remaining panels show the Chebyshev--Lobatto case and
  the four perturbation scales separately, with a common vertical scale.}
  \label{fig:lebesguefunction}
\end{figure}

Figure~\ref{fig:lebesguefunction} shows where the amplification measured by
the Lebesgue constant occurs. For equally spaced nodes, the Lebesgue function
develops large endpoint peaks, which is the familiar mechanism behind the
instability of equally spaced interpolation. In contrast, the
Chebyshev--Lobatto Lebesgue function is much flatter and remains on a
logarithmic scale. The perturbations at and below the logarithmic stability
scale have nearly the same profile as the unperturbed Chebyshev--Lobatto
points. The larger \(0.4n^{-0.8}\) perturbation
produces more pronounced peaks, consistent with the larger values of
\(\Lambda_n\) in Figure~\ref{fig:lebesgue-vs-n}. Thus the function-level plots
support the same conclusion as the constants: the issue is not merely the size
of the largest value, but the emergence of localized peaks in the Lebesgue
function as the angular perturbations become larger.

\subsection{Interpolants of the Runge function}

The next experiment returns to the classical Runge function
\[
    f(x)=\frac{1}{1+25x^2}.
\]
For each degree \(n\), we compare its polynomial interpolants on equally
spaced nodes, on the Chebyshev--Lobatto nodes, and on one random perturbation
of the Chebyshev--Lobatto nodes at the logarithmic scale. The perturbed nodes
are again formed in the angular variable,
\[
    x_j^\varepsilon
    =
    \cos\left(\frac{j\pi}{n}+\varepsilon_j\right),
    \qquad
    \varepsilon_0=\varepsilon_n=0,
    \qquad
    \varepsilon_j\sim
    {\rm Unif}\left[-\frac{1}{n(\log n+1)},\frac{1}{n(\log n+1)}\right].
\]
Thus the perturbation has the same asymptotic order as the deterministic
scale, but with constant one; the explicit theorem requires a sufficiently
small constant in this scale. The interpolants are evaluated on a dense uniform
plotting grid using the barycentric formula with the same
logarithmically scaled weights as above. For the equally spaced case, the
usual explicit binomial form of the barycentric weights is used.

\begin{figure}[htbp]
  \centering
  \includegraphics[width=\textwidth]{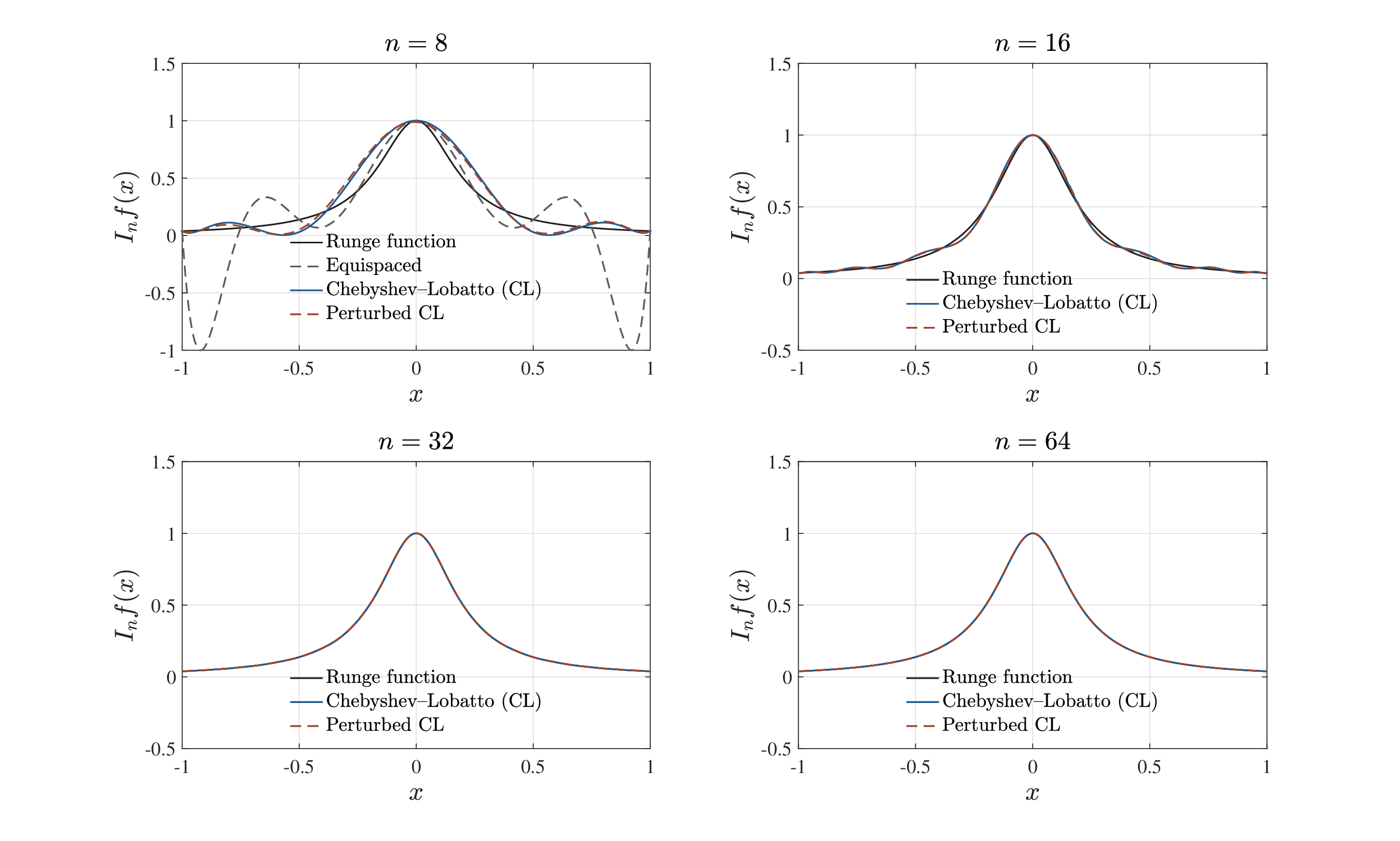}
  \caption{Interpolants of the Runge function
  \(f(x)=1/(1+25x^2)\). For the smallest degree shown, the equally spaced
  interpolant is included to display the onset of the classical endpoint
  oscillation. For the larger degrees, the comparison focuses on the
  Chebyshev--Lobatto interpolant and an angular perturbation with
  \(\sigma_n=(n(\log n+1))^{-1}\).}
  \label{fig:runge-perturbed-cl}
\end{figure}

Figure~\ref{fig:runge-perturbed-cl} connects the Lebesgue-constant behaviour
with the shape of the interpolating polynomials themselves. Equally spaced
interpolation is vulnerable to the large endpoint peaks already visible in
the Lebesgue function: even for the Runge function, whose singularities are
fixed away from the real interval, the polynomial may develop large endpoint
oscillations as \(n\) increases. The Chebyshev--Lobatto interpolants avoid
this mechanism by placing many nodes near the endpoints. The stable perturbed
Chebyshev--Lobatto interpolants are visually indistinguishable from the
unperturbed ones at this scale, apart from small random changes in the curve.
This is consistent with Corollary~\ref{cor:no-runge-analytic}: the
Bernstein-ellipse approximation error decays geometrically, and a
logarithmically growing Lebesgue constant cannot overturn that decay. In this
analytic setting, perturbations at this logarithmic scale therefore do
not exhibit Runge-type divergence.

\subsection{Differentiation errors}

The last experiment illustrates the pseudospectral consequence discussed in
Section~\ref{sec:pseudospectral}. We take
\[
    f(x)=e^x\sin(3x),
\]
so that \(f\), \(f'\), and \(f''\) are known explicitly. For the node set
\(\mathcal X_n^\varepsilon=\{x_j^\varepsilon\}_{j=0}^n\), we compute the
interpolant \(I_n^\varepsilon f\), the first differentiation matrix
\[
    D_{ij}^{(1),\varepsilon}
    =
    \frac{w_j}{w_i(x_i-x_j)},\qquad i\ne j,
    \qquad
    D_{ii}^{(1),\varepsilon}
    =
    -\sum_{j\ne i}D_{ij}^{(1),\varepsilon},
\]
where \(w_j\) are the barycentric weights, and set
\[
    D^{(2),\varepsilon}=(D^{(1),\varepsilon})^2.
\]
Table~\ref{tab:pseudospectral-errors} reports the errors for \(m=0,1,2\). The
interpolation error is evaluated on a dense uniform grid, while the derivative
errors are evaluated at the interpolation nodes:
\[
    \|f-I_n^\varepsilon f\|_\infty,\qquad
    \max_{x_i\in\mathcal X_n^\varepsilon}
    |f'(x_i)-(I_n^\varepsilon f)'(x_i)|,
\]
and
\[
    \max_{x_i\in\mathcal X_n^\varepsilon}
    |f''(x_i)-(I_n^\varepsilon f)''(x_i)|.
\]
The random perturbations are again angular perturbations with
\[
    \varepsilon_j\sim {\rm Unif}[-\sigma_n,\sigma_n], 
    \qquad
    \varepsilon_0=\varepsilon_n=0,
\]
using the same four scales as in Figure~\ref{fig:lebesgue-vs-n}. The random
entries in the table are medians over eight independent perturbations; the
Chebyshev--Lobatto row is deterministic.

\begin{table}[htbp]
\centering
\small
\caption{Interpolation and pseudospectral differentiation errors for
\(f(x)=e^x\sin(3x)\). The columns \(m=0,1,2\) correspond respectively to
\(\|f-I_n^\varepsilon f\|_\infty\), the first derivative error at the nodes, and
the second derivative error at the nodes. Random entries are medians over eight
independent angular perturbations.}
\label{tab:pseudospectral-errors}
\begin{tabular}{r l c c c}
\toprule
$n$ & nodes & $m=0$ & $m=1$ & $m=2$ \\
\midrule
32 & Chebyshev--Lobatto & $1.89\times 10^{-15}$ & $1.45\times 10^{-13}$ & $4.14\times 10^{-11}$ \\
 & $(n(\log n+1))^{-1}$ & $1.78\times 10^{-15}$ & $1.22\times 10^{-13}$ & $3.04\times 10^{-11}$ \\
 & $n^{-1.2}$ & $2.33\times 10^{-15}$ & $1.95\times 10^{-13}$ & $2.89\times 10^{-11}$ \\
 & $n^{-1}$ & $2.90\times 10^{-15}$ & $2.10\times 10^{-13}$ & $8.76\times 10^{-11}$ \\
 & $0.4n^{-0.8}$ & $2.64\times 10^{-15}$ & $1.49\times 10^{-13}$ & $8.04\times 10^{-11}$ \\
 \midrule
64 & Chebyshev--Lobatto & $4.44\times 10^{-15}$ & $1.51\times 10^{-12}$ & $2.64\times 10^{-9}$ \\
 & $(n(\log n+1))^{-1}$ & $3.89\times 10^{-15}$ & $8.18\times 10^{-13}$ & $1.16\times 10^{-9}$ \\
 & $n^{-1.2}$ & $4.94\times 10^{-15}$ & $5.61\times 10^{-13}$ & $1.07\times 10^{-9}$ \\
 & $n^{-1}$ & $9.77\times 10^{-15}$ & $7.13\times 10^{-13}$ & $1.31\times 10^{-9}$ \\
 & $0.4n^{-0.8}$ & $8.02\times 10^{-15}$ & $6.45\times 10^{-13}$ & $6.59\times 10^{-10}$ \\
\midrule
128 & Chebyshev--Lobatto & $7.11\times 10^{-15}$ & $1.46\times 10^{-12}$ & $1.00\times 10^{-8}$ \\
 & $(n(\log n+1))^{-1}$ & $6.72\times 10^{-15}$ & $1.81\times 10^{-12}$ & $1.16\times 10^{-8}$ \\
 & $n^{-1.2}$ & $6.99\times 10^{-15}$ & $2.15\times 10^{-12}$ & $1.49\times 10^{-8}$ \\
 & $n^{-1}$ & $1.98\times 10^{-14}$ & $2.91\times 10^{-12}$ & $2.03\times 10^{-8}$ \\
 & $0.4n^{-0.8}$ & $2.38\times 10^{-14}$ & $3.53\times 10^{-12}$ & $3.16\times 10^{-8}$ \\
\bottomrule
\end{tabular}
\end{table}

Table~\ref{tab:pseudospectral-errors} is intended as a finite-precision
illustration of Theorem~\ref{thm:pseudospectral-differentiation}. Passing from
interpolation to one or two derivatives introduces the expected polynomial
loss in \(n\), so the derivative errors are larger than the interpolation
errors and eventually reflect roundoff and differentiation conditioning. The
perturbation scale \(\sigma_n=(n(\log n+1))^{-1}\), which has the same
asymptotic order as the deterministic stability scale, remains close to the
unperturbed Chebyshev--Lobatto behaviour, while the larger perturbation
\(0.4n^{-0.8}\) gives somewhat larger and less regular errors, consistent with
the Lebesgue-constant behaviour in Figure~\ref{fig:lebesgue-vs-n}.

It is not significant if, for some values of \(n\), a stable perturbed grid
slightly outperforms the unperturbed Chebyshev--Lobatto grid. The theorem is a
stability result, not an optimality statement for every fixed function and
every finite degree. Chebyshev--Lobatto points give a near-minimal
interpolation amplification mechanism in a uniform sense, but the actual error
for a particular analytic function also depends on the alignment of the nodes
with the oscillation and endpoint behaviour of that function. A small random
angular perturbation may therefore improve the finite-\(n\) constant, even
though the asymptotic mechanism and the order of stability remain the same.
The relevant comparison is the envelope: stable perturbations stay close to
the Chebyshev--Lobatto behaviour, whereas perturbations beyond the stable
scale can produce noticeably larger and less predictable errors.

\section{Concluding remarks}
\label{sec:concluding-remarks}

The results of this paper point to a simple organizing principle: for
Chebyshev--Lobatto interpolation, perturbation stability is most naturally
measured in the angular variable.  In that variable the Chebyshev--Lobatto
points form a uniform grid, and the trigonometric Bernstein inequality gives
the factor \(n\) that drives the estimate.  This is why the scale
\(n\sigma_n(\log n+1)\lesssim 1\), with a sufficiently small constant, appears
in the deterministic theorem.  Once the Lebesgue constant remains logarithmic,
analytic interpolation is governed by the balance between the logarithmic
amplification of the interpolation
operator and the geometric decay supplied by Bernstein ellipses.  The numerical
experiments suggest that typical random perturbations may remain stable beyond
the present worst-case scale, but identifying such a probabilistic threshold is
a different problem from the deterministic estimate proved here.

There are several natural directions in which this angular viewpoint might be
tested.  For Legendre points, the error behaviour of Chebyshev and Legendre
approximation has been studied in detail
\cite{Wang2018LegendreExpansion,Wang2023LegendreApproximation,WangXiang2012,XiangChenWang2010}, but perturbation stability should require a
geometry adapted to the Legendre distribution rather than the cosine
parametrization used here.  In two dimensions, Padua points provide an analogue
of Chebyshev-like interpolation on the square, with near-optimal Lebesgue
growth and a structure tied to Lissajous curves
\cite{BosCaliariDeMarchiVianelloXu2006Padua,BosDeMarchiVianelloXu2007Padua}.
Understanding perturbations of this two-dimensional angular structure is a
natural open problem.

Finally, the same question can be asked for related approximation operators.
For interpolatory quadrature on perturbed Chebyshev--Lobatto points, recomputed
weights retain degree-\(n\) exactness and inherit error control from
interpolation, whereas using unperturbed Clenshaw--Curtis weights at perturbed
nodes generally destroys polynomial exactness.  For hyperinterpolation
\cite{an2025path,sloan1995polynomial} and quadrature-based Galerkin schemes
\cite{Wu2026MZGalerkin}, the analogue of the Lebesgue constant is a discrete
norm or Marcinkiewicz--Zygmund stability condition.  These examples point to
the same broader issue: the useful notion of perturbation is dictated by the
sampling geometry and by the operator norm that controls the approximation
process.

\bibliographystyle{siamplain}
\bibliography{myref}

\end{document}